\documentclass[conference,10pt]{IEEEtran}
\IEEEoverridecommandlockouts
\makeatletter
\def\ps@headings{%
\def\@oddhead{\mbox{}\scriptsize\rightmark \hfil \thepage}%
\def\@evenhead{\scriptsize\thepage \hfil \leftmark\mbox{}}%
\def\@oddfoot{}%
\def\@evenfoot{}}
\makeatother 
\usepackage{url,amsmath,graphicx,amssymb,stfloats,subfigure}
\usepackage{mathrsfs,epsfig,epsf,psfrag,amssymb,amsfonts,latexsym,slashbox,graphicx,bm,cite,url}
\usepackage[dvips]{color}

\newtheorem{theorem}{Theorem}

\newtheorem{lemma}{Lemma}

\newtheorem{corollary}{Corollary}

\begin{document}
\title{Delay Optimal Multichannel Opportunistic Access
\thanks{\scriptsize $^*$S. Chen and L. Tong are with the School of Electrical and Computer
Engineering, Cornell University, Ithaca, NY 14853, USA. Email:
{\tt\{sc933@,ltong@ece.\}cornell.edu}.\hfill \break$^\dagger$Q. Zhao is with
the Department of Electrical and Computer Engineering, The
University of California, Davis, CA 95616. Email: {\tt
qzhao@ucdavis.edu}. \hfill \break
The work of S. Chen and L. Tong is
supported in part by the National Science Foundation under award CCF
1018115 and the Army Research Office MURI Program under award
W911NF-08-1-0238. The work of Q. Zhao is supported in part by the
Army Research Laboratory under Grant DAAD19-01-C-0062.}}
\author{Shiyao Chen$^*$, Lang Tong$^*$, and Qing Zhao$^\dagger$}

\maketitle

\begin{abstract}
The problem of minimizing queueing delay of opportunistic
access of multiple continuous time Markov channels is
considered.  A new access policy based on
myopic sensing and adaptive transmission (MS-AT)
is proposed.  Under the framework of risk sensitive constrained
Markov decision process with effective bandwidth as a measure of
queueing delay, it is shown that MS-AT achieves simultaneously throughput and delay optimality.
It is shown further that both the effective bandwidth and the throughput of MS-AT are
two-segment piece-wise linear functions of
the collision constraint (maximum allowable conditional collision probability)
with the effective bandwidth and throughput coinciding in the regime of
tight collision constraints.  Analytical and simulations comparisons with the
myopic sensing and memoryless transmission (MS-MT) policy which  is throughput
optimal but delay suboptimal in the regime of tight collision constraints.
\\[0.2em]
{\em Index terms}---Delay optimal medium access, effective bandwidth, opportunistic access, and constrained risk sensitive Markov decision process.
\end{abstract}

\section{Introduction}
\IEEEPARstart{W}{e} consider in this paper delay optimal policies for a secondary user in a multichannel hierarchical overlay cognitive network \cite{Zhao&Sadler:07SPM}, where a cognitive user can sense and transmit on one of the $N$ channels assigned to the primary users.  A secondary user should only transmit on a channel where the primary user is not transmitting, and  its transmissions are subject to interference constraints imposed by primary incumbents.

In such hierarchical cognitive networks, transmission opportunities for a secondary user  depend on random traffic patterns of the primary users, which makes it necessary that the secondary user has a queue that holds the arrival packets. Thus packets of a secondary transmitter are subjects to random delays. For applications with delay constraints, there is a need to find a sensing and transmission policies that minimize  queueing delays caused by random transmission opportunities and transmission failures due to collisions with the primary users.

Effective bandwidth (or effective capacity in the terminology used in
\cite{Wu:03,Ying:06,Shakkottai:TAC08}) is an indirect way to
measure the delay performance of a policy for applications either with
bursty random arrivals at or random departures from the transmitter queue; the latter may be caused by channel fading as in \cite{Wu:03,Ying:06} or random transmission opportunities considered in this paper.
To characterize the delay due to random arrivals or departures, Kelly
argued in \cite{Kelly:91} that ``\textit{constraints on the probability that buffer space or
delay exceeds a certain threshold is more important than constraints
on mean values.}" The idea is that a provider offering services
 should guarantee the subscriber a measure of  ``effective bandwidth" that takes into account randomness either in the arrival or the departure processes.
This viewpoint, implicitly espoused earlier in \cite{Hui:88JSAC}, led to the class of approaches based on the theory of large deviations.  See \cite{Elwalid&Mitra:93TON,Glynn:94JAP,Veciana&Walrand:95QS,Shwartz&Weiss:95book,Kelly:96} and references therein.

In the context of opportunistic secondary transmissions, the notion of effective bandwidth is defined as follows.
Given the buffer size $b\gg 1$ of the secondary transmitter and $\epsilon$, the
effective bandwidth is the maximum arrival rate  $a(\epsilon, b)$ of the secondary user traffic such that, despite the randomness in transmission opportunities and collisions with primary users, the probability of the queue length exceeding  $b$ is capped below $\epsilon$.  For applications with constant rate of arrivals, the queue size of a secondary user is proportional to packet delays. Therefore the effective bandwidth can be interpreted as the sustainable arrival rate with delay constraint.

On the surface, transmission delays caused random opportunities resembles that caused by fading, which has been studied in \cite{Wu:03,Ying:06,Shakkottai:TAC08}. There are, however, some important differences.  What we are interested in are {\em delay optimal policies} that choose which channel to sense and the way to transmit; the problem of finding an optimal policy was not the objective of \cite{Wu:03,Ying:06,Shakkottai:TAC08}. In addition, the collision constraints imposed in this paper do not have a correspondence in the problems studied earlier.

\subsection{Summary of Results and Contribution}

We consider $N$ identical and independent primary channels, modeled
by on-off continuous time Markov processes, with ``on" indicating
the channel is being used by the primary user and ``off" a transmission opportunity
for the secondary user. The secondary user adopts a slotted sensing-before-transmission protocol which
is defined by a sensing policy that determines which channel to sense and a transmission policy
that specifies the probability of transmission.

If the secondary user transmits on a particular channel, a collision with the primary user is possible because, as the licensed incumbent, the primary user transmits whenever it has a packet. We assume that the secondary user abides by the interference constraints imposed by the primary users.
Here the interference to a primary user is measured by the collision probability between the primary and secondary users conditioned on the event that the primary user transmits.
In this paper, the set of admissible policies includes those whose interference to primary channels is uniformly bounded above by  $\gamma$.

The main results of this paper are as follows. First, we propose a simple policy referred to
as MS-AT (myopic sensing and adaptive transmission). The sensing policy is the myopic policy originally
proposed by Zhao, Krishnamachari, and Liu
\cite{Zhao&Krishnamachari&Liu:08TWC} for slotted primary user systems.
The key idea of MS-AT is an adaptive transmission policy that adjust adaptively its transmission probabilities based on  feedbacks of past transmissions (ACK/NAK). In particular, given the maximum acceptable collision constraint $\gamma$,
the adaptive transmission policy sets a target rate $\tau(\gamma)$ of successful transmissions and adjust its transmission probabilities based on the estimated rate of successful transmissions.  We note that
the proposed MS-AT policy can be interpreted as a debt-based policy whose principle is  used  to
establish feasible throughput optimality in the deadline scheduling problem  \cite{Hou&Borkar&Kumar:INFOCOM09,Hou&Kumar:MobiHoc09,Hou&Kumar:INFOCOM10}, though the deadline scheduling  problem and the debt-based policy considered in \cite{Hou&Borkar&Kumar:INFOCOM09,Hou&Kumar:MobiHoc09,Hou&Kumar:INFOCOM10}
are very different from those in this paper.

Second, we establish  that MS-AT  simultaneously achieves delay and throughput optimality under any collision constraint. In establishing the optimality of MS-AT, the main idea is to formulate the problem of maximizing
effective bandwidth subject to collision constraints as one of risk sensitive Markov decision
process which adopts an exponential objective function (instead of the usual
linear objective function in the risk neutral Markov decision process).  In particular, in
the regime of loose collision constraints, the optimality of MS-AT is shown by the technique of
Hernandez and Marcus who established the multiplicative dynamic programming
representation of the unconstrained risk sensitive Markov decision process in
\cite{Hernandez&Marcus:96SCL}.

Third, we   provide simple characterizations of the
throughput and effective bandwidth as functions of the collision constraint parameter $\gamma$.  It
is shown that both the effective bandwidth and the throughput of MS-AT are
two-segment piece-wise linear functions of $\gamma$, and the effective bandwidth and throughput coincide
in the regime of tight collision constraints as illustrated in Fig.~\ref{fig:TputEB}.

\begin{figure}
\centering
\begin{psfrags}
\includegraphics[width=2.7in]{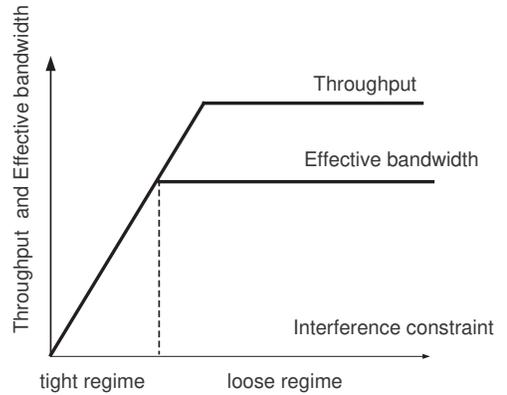}
\end{psfrags}
\caption{Illustration of throughput and effective bandwidth of MS-AT as piece-wise linear functions of collision constraints.} \label{fig:TputEB}
\end{figure}

The  result that the effective bandwidth matches to the throughput in the regime of tight collision constraints
seems surprising since the effective bandwidth is typically strictly smaller than the throughput \cite{Chang:00book}. A commonly accepted intuition is that the randomness in the input process requires more bandwidth, and randomness in the output process decreases the sustainable rate traffic if the delay is required to be small.
This intuition is consistent with our result when the collision constraints are loose as it can be observed in Section \ref{sec:simu} that under loose collision constraints the optimal effective bandwidth is strictly smaller than the optimal throughput. However, when the collision constraint is tight, the delay optimal policy offers effective bandwidth equal to the optimal throughput (see Theorem \ref{thm:tightthm}).
The key insight, as presented in Section \ref{sec:four}, is that the secondary user can make its packet departure process favorable by an adaptive transmission policy.

\subsection{Related Work}

Substantial amount of work exists for opportunistic access with
throughput as the performance measure. In
\cite{Zhao&Tong&Swami&Chen:JSAC07}, Zhao \textit{et. al.} consider
the case when primary users follow a slotted Markov transmission
structure under the per slot interference constraint. The myopic
sensing policy is shown to be throughput optimal for identical,
positively correlated Markovian channels
\cite{Zhao&Krishnamachari&Liu:08TWC,Ahmad&Liu&Javidi&Zhao&Krishnamachari:09IT}.
The continuous time Markovian occupancy model has also been treated
using the throughput measure (see, \textit{e.g.}, \cite{Geirhofer&Tong&Sadler:07COMMag,%
Geirhofer&Tong&Sadler:08MILCOM,
Zhao&Geirhofer&Tong&Sadler:08SP}). Adopting a periodic channel
sensing policy, the optimal transmission policy is obtained in the
framework of constrained Markov decision process under the average interference
constraint
\cite{
Zhao&Geirhofer&Tong&Sadler:08SP}, which further leads to the
throughput optimality of the periodic sensing with memoryless
transmission (PS-MT) policy when the collision constraints are tight\footnote{the tight collision constraint in \cite{Li&Zhao&Guan&Tong:10ICC,Li&Zhao&Guan&Tong:09JSAC} is different from the tight collision regime for MS-MT}
\cite{Li&Zhao&Guan&Tong:10ICC,Li&Zhao&Guan&Tong:09JSAC}.
The multiuser maximum throughput region of opportunistic access is
studied in \cite{Chen&Tong:JSAC} in which the throughput optimal policies employ memoryless transmission policy.

The use of effective bandwidth to analyze delay for opportunistic spectrum access was first considered in \cite{Laourine&Chen&Tong:10INFOCOM} where the authors obtain in closed
form the effective bandwidth of the myopic sensing policy with
slotted primary transmission, without establishing the optimality of
myopic sensing. In \cite{Chen&Tong:10ITA} both throughput and
effective bandwidth are considered for multi-user cognitive access.
The effective bandwidth optimality of myopic sensing for two
identical Markov channels under loose collision constraints is obtained in \cite{Chen&Tong:10ITA},
whereas our result in this paper establishes the optimality for a
general number of homogeneous channels under loose collision constraint as well as new optimal policy for general collision constraint.
In \cite{Wang&etal:INFOCOM10} the authors propose an ALOHA based
policy for multiuser opportunistic access and analyze its delay performance via a fluid model.

The optimization of effective bandwidth under tight collision
constraint can be formulated as constrained risk sensitive Markov decision. The only existing results to our
knowledge are \cite{Sladky:09,Kadota:06}, where \cite{Sladky:09} considers the set of Markov policies whereas \cite{Kadota:06}
treats the general constrained Markov decision process with the objective function and
the constraints given by general utility functions and establishes the existence of an optimal policy in the
set of general history dependent policies without any structural property or computation procedure of the optimal policy. In this paper, we exploit the special form of the utility function and the associated linear constraints, which leads to a structured optimal policy.

\section{Network Model}\label{sec:two}

The network model and the assumptions are described as follows. The
hierarchical overlay cognitive network consists of $N$ parallel
primary channels indexed by $1,\ldots,N$ and a secondary user. Each
primary user transmits on its dedicated channel, the occupancy of
which is assumed to be independent, identically distributed continuous time Markov process. The
state space of the channel is $\{1\mbox{(busy)},0\mbox{(idle)}\}$
and the holding times are exponentially distributed with mean
$\mu^{-1}$ and $\lambda^{-1}$ for busy and idle states,
respectively. The secondary user opportunistically accesses the
primary channels in a slotted sensing-before-transmission manner
with slot length $T$. In each time slot the secondary user senses
one out of the $N$ channels (sensing is assumed to be perfect) and
decides whether or not to transmit in the sensed channel. Successful
transmission of $c$ bits is achieved in slot $t$ if the secondary
user transmits in a channel that is idle throughout slot $t$.
In particular, if the secondary user senses a channel to be busy in slot $t$, no transmission will be attempted in this slot since no successful transmission but collision will be achieved.
We aim to maximize the effective bandwidth subject to the collision
constraints imposed by the primary users.

\subsection{Effective Bandwidth}

Effective bandwidth characterizes the quantity of service available
to the secondary user via the opportunistic channel with QoS
constraint. To give the definition of the effective bandwidth we
consider the queueing process of the secondary user.

Assume that the incoming traffic of the secondary user is a constant
arrival process with intensity of $a$ bits per slot, and the bits
are stored in a buffer of size $b\gg1$ before being transmitted. For
a fixed opportunistic access policy $\pi$, denote by $Q^\pi_{t}$ the queue length
at the end of slot $t$. Then $(Q^\pi_{t})_{t\geq0}$ satisfies the
following recursion
\begin{equation*}
Q^\pi_{t}=\max\{Q^\pi_{t-1}+a-R^\pi_{t},0\},~~t\geq1,~~Q^\pi_{0}=0.
\end{equation*}
where $(R^\pi_t)_{t\geq0}$ is the output process of the
opportunistic channels ($R^\pi_t=c$ if successful transmission is
achieved in slot $t$, and $R^\pi_t=0$ otherwise).
Under regularity conditions, Glynn and Whitt show in \cite{Glynn:94JAP} that, if the policy $\pi$ is such that the queue is stable, then $Q_n^\pi$ converges in distribution to a steady state distribution $Q_\infty^\pi$ with the decay rate
\begin{equation}\label{eq:decayrate}
\lim_{x\rightarrow\infty}\frac{1}{x}\log\Pr(Q^\pi_{\infty}>x)=-\theta^\pi(a),
\end{equation}
where $\theta^\pi(a)$ is the unique positive solution of the equation
\begin{equation}\label{eq:eq}
a\theta+\Psi_R^\pi(-\theta)=0,
\end{equation}
and $\Psi_R^\pi(\theta)$ is the Gartner-Ellis limit
\begin{equation*}
\Psi^\pi_R(\theta)=\lim_{n\rightarrow\infty}\frac{1}{n}\log \mathbb{E}_\pi[\exp(\theta
\sum_{k=1}^{n}R^\pi_{k})].
\end{equation*}

The decay rate $\theta^\pi(a)$ in Eq. (\ref{eq:decayrate})
depends on the secondary user's traffic arrival rate $a$ and the
output process $(R_t^\pi)_{t\geq0}$; the latter depends on the primary channel statistics and the opportunistic access
policy. With the decay rate in Eq. (\ref{eq:decayrate}) the packet drop (buffer overflow)
probability can be approximated by
\begin{equation}\label{eq:approx}
\Pr[Q^\pi_{\infty}>b]\approx\exp(-\theta^\pi(a)b),
\end{equation}
for large buffer size $b$.

With the constraint that the buffer overflow probability is capped by $\epsilon$ and adopting the large buffer approximation in Eq.
(\ref{eq:approx}), the effective bandwidth $a^\pi(\epsilon,b)$ can be
defined as
\begin{equation}\label{eqn:ebdef}
a^\pi(\epsilon,b)\triangleq\max\{a:
\exp(-b\theta^\pi(a))\leq\epsilon\}.
\end{equation}
It can then be shown (see \cite{Laourine&Chen&Tong:10INFOCOM}) that the
following effective bandwidth
formula holds.
\begin{equation}\label{eq:eb}
a^\pi(\epsilon,b)=\lim_{n\to\infty}\frac{\log\mathbb{E}_\pi\exp(\frac{\log(\epsilon)}{b}\sum_{t=1}^nR^\pi_t)}{n\frac{\log(\epsilon)}{b}}.
\end{equation}

The above expression shows a key connection between effective bandwidth and throughput. Applying the Jensen's inequality, we have
$a^\pi(\epsilon,b)\leq\lim_{n\to\infty}\frac{1}{n}\mathbb{E}_\pi\sum_{t=1}^nR^\pi_t$.
In general, the effective bandwidth is strictly smaller than the throughput, except in certain special scenarios (see Section \ref{sec:four}). Asymptotically as the
buffer size approaches infinity ($b\rightarrow \infty$) or the QoS
requirement is relaxed ($\epsilon\rightarrow1$), the effective bandwidth does converge to the throughput.

\subsection{Collision Constraints}\label{sec:collision}

The transmissions of the secondary user are subject to collision
constraints imposed by the primary users.
The collision is measured in a long term
average manner. Specifically, for each primary user $i$ the scaled infinite horizon
average collision with primary user $i$ must be bounded by $\gamma$; the scale infinite horizon average collision for primary user
$i$ is
\begin{equation*}
C_{\pi,i}=\frac{1}{1-v(0)\exp(-\lambda
T)}\lim_{n\to\infty}\frac{1}{n}\mathbb{E}\sum_{t=1}^n\chi_{i,t},
\end{equation*}
where $\chi_{i,t}$ is the indicator of collision with primary user $i$
in slot $t$, and the scale uses the reciprocal of
$1-v(0)\exp(-\lambda T)$ ($v(0)=\mu/(\lambda+\mu)$ is the stationary distribution of
idle state), \textit{i.e.}, the steady state probability of primary
user $i$ transmitting in a certain slot.
Given the collision parameter $\gamma$, the set of admissible policies
$\boldsymbol{\Pi}(\gamma)$ is given by the set of
policies that meet the interference constraints, \textit{i.e.},
$\boldsymbol{\Pi}(\gamma)=\{\pi:C_{\pi,i}\leq\gamma,\forall i\}$.
We aim to characterize the maximum effective bandwidth for the admissible policy set $\boldsymbol{\Pi}(\gamma)$.

\subsection{Opportunistic Access Policy}

The opportunistic access policy consists of two components: the
channel sensing policy which selects a channel to sense in each slot
based on the past history, and the transmission policy which
specifies the transmission probability upon idle sensing results. (Upon busy sensing results the secondary user will keep silent.)

We first describe the myopic
sensing policy (MS). In
\cite{Zhao&Krishnamachari&Liu:08TWC,Ahmad&Liu&Javidi&Zhao&Krishnamachari:09IT},
myopic sensing is shown to be throughput optimal for independent, identically distributed,
positively correlated, discrete time Markov channels with a simple round-robin
structure. Specifically, the secondary user first fixes an ordered
list of the $N$ channels. To start, the secondary user senses the
first channel in the list and keeps sensing it until the first busy
sensing result. Then the secondary user switches to the next channel
in the list and keeps sensing this channel. The secondary user
cycles through the channel list in this stay-when-idle,
switch-when-busy manner.

When it comes to the transmission policy, the
transmission probability is determined, in general, based on the
entire history.
A simple (but in general suboptimal) transmission policy is the memoryless transmission (MT) policies where the transmission decision depends only on the current sensing outcome with constant probability of transmission upon idle sensing results. We consider in this paper a new class of transmission policy referred to as adaptive transmission (AT) policies. As discussed in detail in Section \ref{sec:four}, the transmission strategies at time slot $t$ depends on the history of transmissions in the past.

In this paper, we consider two types of opportunistic access policies: MS-MT---the myopic sensing and memoryless transmission \cite{Chen&Tong:10ITA,Chen&Tong:JSAC} (with constant transmission probability $\mu_{\tiny \mbox{MS}}$
chosen to satisfy the interference constraints) and a new policy referred to as MS-AT---myopic sensing and adaptive transmission.
Figs. \ref{fig:illu1} illustrates the sample path of the MS-MT policy and the MS-AT policy is described in Section \ref{sec:four}.

\begin{figure}
\centering
\begin{psfrags}
\psfrag{PU}[l]{\footnotesize PU Tx.}
\psfrag{CU}[l]{\footnotesize CU Tx.}
\psfrag{I}[l]{\footnotesize idle sensing}
\psfrag{B}[l]{\footnotesize busy sensing}
\psfrag{ch1}[c]{\footnotesize Channel 1} \psfrag{ch2}[c]{\footnotesize
Channel 2}\psfrag{ch3}[c]{\footnotesize Channel 3}
\includegraphics[width=3in]{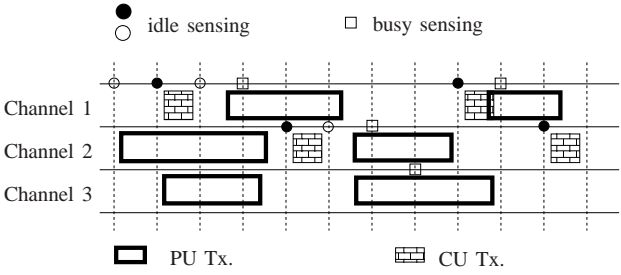}
\end{psfrags}
\caption{Illustration of MS-MT policy. Filled(Open) circle: secondary user decides
to(not to) transmit.} \label{fig:illu1}
\end{figure}

\section{Optimality of MS-AT}\label{sec:four}

First define $\theta=\log(\epsilon)/b$ to be the
effective bandwidth parameter for buffer size $b$ and packet drop
probability $\epsilon$. Rewrite the effective bandwidth
formula (\ref{eq:eb}) with $\theta$
\begin{equation*}
a^\pi(\epsilon,b)=\lim_{n\to\infty}\frac{\log\mathbb{E}_\pi\exp(\theta\sum_{t=1}^nR^\pi_t)}{n\theta}.
\end{equation*}

\subsection{Adaptive Transmission Policy}

According to the effective bandwidth
formula (\ref{eq:eb}) a simple upper bound can be easily obtained, \textit{i.e.}, the optimal throughput under collision parameter $\gamma$. To achieve this upper bound, we first describe the mechanism of adaptive transmission
policy (AT). We employ the acknowledgement of the secondary user to
aid the transmission decision. Specifically, after sensing a channel to be idle in slot $t$ the secondary user
counts the total number of acknowledgements received up to slot
$t-1$, denoted by $A_t$ and transmits with probability
$1$ if $A_t<\tau t$. Otherwise if $A_t\geq\tau t$, the secondary
user stays silent. Here $\tau$ is the parameter of AT policy
and controls the collision caused by the secondary user.

The adaptive transmission policy can be also thought of as a debt based transmission policy. The target of the secondary user is set to be transmitting $\tau t$ packets within the first $t$ slots, thus if $A_t<\tau t$, the secondary user is in debt and thus needs to transmits with probability
$1$, and otherwise, if $A_t\geq\tau t$, the secondary
user has a reasonable balance and stays silent.

The optimality of the MS-AT policy with respect to
effective bandwidth and the optimal effective bandwidth as a function of the collision parameter $\gamma$ is provided in Theorem \ref{thm:tightthm}.

\begin{theorem}\label{thm:tightthm}
For any collision parameter $\gamma$, there
exists $\tau$ such that the MS-AT policy with parameter $\tau$ is feasible and
optimal for the constrained
effective bandwidth optimization with collision parameter
$\gamma$. The parameter $\tau$ of the optimal MS-AT policy is given by
\begin{equation}
\tau=\frac{N\gamma\exp(-\lambda T)(1-v(0)\exp(-\lambda
T))}{1-\exp(-\lambda
T)},
\end{equation}
where $v(0)=\mu/(\lambda+\mu)$ is the stationary distribution of
idle state. The optimal effective bandwidth is given by
\begin{equation}\label{eqn:opteb}
\mbox{\textmd{EB}}^\ast=\min\{\frac{N\gamma\exp(-\lambda T)(1-v(0)\exp(-\lambda
T))}{1-\exp(-\lambda
T)},\frac{\Psi_X^{\tiny \mbox{MS}(\infty)}(\tilde{\theta})}{\theta}\},
\end{equation}
where \begin{equation}\label{eqn:tilde}
\tilde{\theta}=\log(1-\exp(-\lambda T)(1-\exp(c\theta))),
\end{equation}
\begin{equation*}
\Psi^{\tiny \mbox{MS}(\infty)}_X(\tilde{\theta})=\lim_{n\rightarrow\infty}\frac{1}{n}\log \mathbb{E}^{\tiny \mbox{MS}(\infty)}[\exp(\tilde{\theta}
\sum_{k=1}^{n}R^{\tiny \mbox{MS}(\infty)}_k)],
\end{equation*}
and $\Psi_X^{\tiny \mbox{MS}(\infty)}(\tilde{\theta})/\theta$ is the optimal effective bandwidth when the collision parameter $\gamma$ is loose (the minimum in Eq. (\ref{eqn:opteb}) is assumed by the second term).

Also the MS-AT policy with parameter $\tau$ is
optimal with respect to the throughput, with optimal throughput
\begin{eqnarray*}
\mbox{\textmd{TH}}^\ast&=&\min\{\frac{N\gamma\exp(-\lambda T)(1-v(0)\exp(-\lambda
T))}{1-\exp(-\lambda
T)},\\
&&\lim_{n\to\infty}\frac{1}{n}\mathbb{E}^{\tiny \mbox{MS}(\infty)}\sum_{t=1}^nR^{\tiny \mbox{MS}(\infty)}_t\}.
\end{eqnarray*}
\end{theorem}

Theorem \ref{thm:tightthm} provides the solution to the constrained
effective bandwidth optimization problem by establishing the
optimality of MS-AT for all collision parameter $\gamma$. To prove Theorem \ref{thm:tightthm}, we
need to first prove a list of lemmas. In the later presentation we assume $c=1$ without loss of generality.

\subsection{MS-AT with $\tau=\infty$}

We first establish the effective bandwidth optimality of
MS-AT policy with $\tau=\infty$ under loose collision parameter $\gamma$. We denote the MS-AT policy with $\tau=\infty$ by MS-AT$(\infty)$. By the structure of MS-AT, the secondary user always transmits upon idle sensing results under MS-AT$(\infty)$ policy.
Denote by $X^{\tiny \mbox{MS}(\infty)}_t$ the indicator of the secondary user having an
idle sensing result in slot $t$, following MS-AT$(\infty)$ policy.

Call a collision parameter $\gamma$ loose if the MS-AT$(\infty)$ policy is admissible with respect to $\gamma$. We are interested in proving that the MS-AT$(\infty)$ policy has better effective bandwidth than any other admissible opportunistic access policy. To this end, we relax the collision constraint and turn to show that the
MS-AT$(\infty)$ policy has better effective bandwidth than any other opportunistic access policy, whether admissible or not. This relaxation leads to an unconstrained risk sensitive Markov decision process, the optimality of which is
provided in Lemma \ref{thm:1}.

\begin{lemma}\label{thm:1}
Under loose collision
constraint, MS-AT$(\infty)$ is effective bandwidth
optimal over all admissible policies in $\boldsymbol{\Pi}(\gamma)$.
\end{lemma}

\begin{proof}
In view of Eq. (\ref{eq:eb}) and the fact that the
effective bandwidth parameter $\theta=\log(\epsilon)/b<0$, to show the effective bandwidth
optimality we just need to show that MS-AT$(\infty)$ minimizes
$M^\pi_{R,n}(\theta)=\mathbb{E}_\pi\exp(\theta\sum_{t=1}^nR^\pi_t)$ over all $\pi\in\Pi(\gamma)$ for all
$n\geq 1$. We then make connection between $M^\pi_{R,n}(\theta)$ and
$M^\pi_{X,n}(\theta)=\mathbb{E}_\pi\exp(\theta\sum_{t=1}^nX^\pi_t)$, where $X^\pi_t$
is the indicator of the secondary user having an idle sensing result
in slot $t$ under opportunistic access policy $\pi$. Specifically, we have the following inequality
\begin{eqnarray}
\nonumber M^\pi_{R,n}(\theta)&=&
\mathbb{E}[\mathbb{E}[e^{\theta\sum_{t=1}^nR^\pi_t}\mid
X^\pi_1,\ldots,X^\pi_n]] \\
\nonumber &=&\mathbb{E}[\prod_{t=1}^n\mathbb{E}[e^{\theta R^\pi_t}\mid X^\pi_t]] \\
\nonumber &=&\mathbb{E}[\prod_{t=1,X^\pi_t=1}^n [1-\mu^\pi_te^{-\lambda
T}(1-e^{c\theta})]] \\
\label{eqn:geq}&\geq&\mathbb{E}[[1-(\sup_{t=1}^\infty\mu^\pi_t)e^{-\lambda
T}(1-e^{c\theta})]^{\sum_{t=1}^nX^\pi_t}] \\
\label{eqn:connection}
&=&M^\pi_{X,n}(\log(1-(\sup_{t=1}^\infty\mu^\pi_t)e^{-\lambda T}(1-e^{c\theta}))),
\end{eqnarray}
where $\mu^\pi_t$ is the transmission probability used by policy $\pi$ in slot $t$.

Under loose collision constraint it holds for opportunistic access policy $\pi$
that $\mu_t^\pi\leq1$. On the other hand, the MS-AT$(\infty)$ policy always transmit with probability 1. Therefore
equality holds in Eq.
(\ref{eqn:geq}) for MS-AT$(\infty)$. Thus we just need to show that for
an arbitrary opportunistic access policy $\pi$,
\begin{equation}\label{eqn:minimize}
M_{X,n}^\pi(\theta)\geq
M_{X,n}^{\tiny \mbox{MS}(\infty)}(\theta),
\end{equation}
for all $\theta<0$.

We prove Eq. (\ref{eqn:minimize})
with the following multiplicative dynamic programming recursion \cite{Hernandez&Marcus:96SCL} for
the finite horizon problem with
$\mathbb{E}e^{\theta\sum_{k=1}^KX_k}$ as the objective function. First define
\begin{eqnarray*}
V_K(\boldsymbol{\omega})&=&\min_{1\leq a\leq
N}\mathbb{E}e^{\theta X_a(\boldsymbol{\omega})}\\
V_t(\boldsymbol{\omega})&=&\min_{1\leq a\leq N}\mathbb{E}[e^{\theta
X_a(\boldsymbol{\omega})}V_{t+1}(\mathcal{T}(\boldsymbol{\omega},a))]
\end{eqnarray*}
for $t=1,\ldots,K-1$, where $a$ is the dummy variable denoting the current action at time $t$ of choosing to sense the $a$th channel, and $V_t(\cdot)$ is the value function,
defined as the minimum expected future objective function that can
be achieved starting from $t$ when the information state is
$\boldsymbol{\omega}$, \textit{i.e.},
$V_t(\boldsymbol{\omega})=\inf_\pi\mathbb{E}_\pi[
\exp(\theta\sum_{k=t}^KX_k)\mid\boldsymbol{\omega}]$, where the information state
$\boldsymbol{\omega}$ is the length $N$ vector with the $i$th component being the conditional probability that the $i$th primary channel will be sensed idle in the next slot (see \cite{Ahmad&Liu&Javidi&Zhao&Krishnamachari:09IT} for the
definition and interpretation of the channel information state
vector $\boldsymbol{\omega}$). Expanding the dynamic programming
recursion based on the sensing outcome, we have
\begin{eqnarray*}
V_K(\boldsymbol{\omega})&=&\min_{1\leq a\leq N}\mathbb{E}e^{\theta
X_a(\boldsymbol{\omega})}=\min_{1\leq a\leq N}\{\omega_ae^{\theta}+(1-\omega_a)\} \\
V_t(\boldsymbol{\omega})&=&\min_{1\leq a\leq
N}\mathbb{E}[e^{\theta X_a(\boldsymbol{\omega})}V_{t+1}(\mathcal{T}(\boldsymbol{\omega},a))] \\
&=&\min_{1\leq a\leq
N}\{\omega_ae^{\theta}V_{t+1}(\mathcal{T}(\boldsymbol{\omega},a\mid 0)) \\
&&+(1-\omega_a)V_{t+1}(\mathcal{T}(\boldsymbol{\omega},a)\mid1)\}
\end{eqnarray*}




We can proceed from here and show that picking the largest component
in the channel information state vector $\boldsymbol{\omega}$ to
sense solves the multiplicative dynamic programming equation, and
thus prove the effective bandwidth optimality of MS-MT, following the line of
lemmas in \cite{Ahmad&Liu&Javidi&Zhao&Krishnamachari:09IT}. One key
lemma in \cite{Ahmad&Liu&Javidi&Zhao&Krishnamachari:09IT} (Lemma 3)
needs to be modified to the multiplicative version
\begin{eqnarray*}
e^\theta V_{t+1}(\omega_1,\ldots,\omega_{n-2},p_{11},p_{01})\geq
V_{t+1}(p_{01},\omega_1,\ldots,\omega_{n-2},p_{11}).
\end{eqnarray*}
We omit the details of the proof due to limited space.
%
\end{proof}

Note that under loose collision constraints as $\gamma$ increases the optimal effective bandwidth stays constant; the optimal value is
\begin{equation*}
\mbox{\textmd{EB}}^\ast(\epsilon,b)=\frac{\Psi_R^{\tiny \mbox{MS}(\infty)}(\theta)}{\theta}
=\frac{\Psi_X^{\tiny \mbox{MS}(\infty)}(\tilde{\theta})}{\theta},
\end{equation*}
where $\theta=\log(\epsilon)/b$, and
\begin{equation}\label{eqn:tilde}
\tilde{\theta}=\log(1-\exp(-\lambda T)(1-\exp(c\theta))),
\end{equation}
and the last equality follows from Eq. (\ref{eqn:connection}).

\subsection{Structural Properties of MS-AT}

We analyze the structure of the adaptive transmission policy
(AT) coupled with myopic channel sensing policy. Specifically, there
will be consecutive time slots in which the secondary user is in good balance (therefore stays
silent) and is in debt (therefore keeps transmitting upon idle sensing results), respectively. We
denote the consecutive intervals by $I_1,\ldots,I_n,\ldots$ and
$B_1,\ldots,B_n,\ldots$. We also denote by $i_1,\ldots,i_n,\ldots$
and $b_1,\ldots,b_n,\ldots$ the first time slot in the intervals
$I_1,\ldots,I_n,\ldots$ and
$B_1,\ldots,B_n,\ldots$. Note that in the first slot it always holds that $A_1=0<\tau$, which indicates that the interval $B_1$ ($B_n$) comes before interval $I_1$ ($I_n$).

\begin{lemma}\label{lemma:mono}
As $\tau$ increases, the effective bandwidth associated with MS-AT with parameter $\tau$ is non-decreasing.
\end{lemma}

\begin{proof}
Prove this lemma with a sample path argument. Specifically, for arbitrary $\tau_1<\tau_2$ and any horizon $n$ we show that along any sample path of the channel state realization the MS-AT policy with $\tau_2$ accumulates more successful transmissions than the MS-AT policy with $\tau_1$, \textit{i.e.},
\begin{equation}\label{eqn:moretransmission}
\sum_{t=1}^nr_t^{\tau_1}\leq\sum_{t=1}^nr_t^{\tau_2},
\end{equation}
where the lower case $r$ indicates the quantities are conditioned on a sample path realization. Eq. (\ref{eqn:moretransmission}) will follow if one can show for any integer $k$ that the MS-AT policy with $\tau_2$ achieves the $k$th successful transmission no later than the MS-AT policy with $\tau_1$ for the specific channel state realization.

Assume otherwise, \textit{i.e}, for some integer $k$ the $k$th successful transmission occurs earlier under the MS-AT policy with $\tau_1$. Without loss of generality assume that the $k$ is the smallest among the integers that satisfy this condition. Thus a contradiction can be easily drawn since in the time slot of the $k$th successful transmission under the MS-AT policy with $\tau_1$, the MS-AT policy with $\tau_2$ has more debt than the MS-AT policy with $\tau_1$, and thus will decide to transmit, which is guaranteed to be successful since the MS-AT policy with $\tau_1$ has the $k$th successful transmission in that very time slot. Therefore the contradiction proves Eq. (\ref{eqn:moretransmission}), and the claim of the lemma follows.
\end{proof}

\begin{lemma}\label{lemma:structure1}
For all $k\geq1$ it holds that
\begin{equation}
A_{i_k}\geq i_k\tau,A_{b_k}<b_k\tau,
\end{equation} and
\begin{equation}
A_{i_k-1}<(i_k-1)\tau,A_{b_k-1}\geq(b_k-1)\tau.
\end{equation}
\end{lemma}

\begin{proof}
We prove the inequalities using the alternating nature of the
sequence of intervals we define and the structure of the MS-AT policy.

By the structure of the adaptive transmission policy, in time slot $t$ if the MS-AT policy attempts to transmit (\textit{i.e.}, $t\in B_k$ for some $k$), then $A_t<\tau t$. Similarly, if the MS-AT policy stays silent (\textit{i.e.}, $t\in I_k$ for some $k$), then $A_t\geq\tau t$.

For all $k\geq1$, by definition $i_k$ is the first time slot in the
$k$th idle period, therefore $A_{i_k}\geq i_k\tau$.
For all $k\geq1$, by definition $i_k-1$ is the last time slot before
the $k$th idle period, therefore $i_k-1$ lies in a certain busy interval, which indicates $A_{i_k-1}<(i_k-1)\tau$.
The other two inequalities can be shown similarly, except that there is one special case for $k=1$ ($b_1-1=0$ does not lie in any idle or busy interval). To accommodate this special case we define by convention $A_0=0$.
\end{proof}

\begin{lemma}
For any $\tau<\Psi_X^{\tiny \mbox{MS}(\infty)}(\tilde{\theta})/\theta$,
\begin{enumerate}
\item the lengths of the intervals
$I_k$ are bounded by $\frac{1}{\tau}$ almost surely,
\item the lengths of the intervals $B_k$ have finite expectation.
\end{enumerate}
\end{lemma}

\begin{proof}
In time slot $1$, the secondary user is always transmitting since $A_1=0<\tau$. Therefore the $k$th idle interval $I_k$ is ahead of the $k$th busy interval $B_k$. Thus the idle interval $I_k$ starts from slot $i_k$ (inclusion) and ends in slot $b_{k+1}-1$ (inclusion). Therefore the length of the idle interval $I_k$ is $|I_k|=b_{k+1}-i_k$.

According to the structure of MS-AT policy and Lemma \ref{lemma:structure1}, it holds that
\begin{equation*}
A_{i_k}\leq A_{i_k-1}+c<(i_k-1)\tau+c,
\end{equation*}
\begin{equation*}
A_{b_{k+1}}\geq A_{b_{k+1}-1}\geq (b_{k+1}-1)\tau,
\end{equation*}
and because during the idle interval $I_k$ nothing is transmitted, it follows that $A_{b_{k+1}}=A_{i_k}$. Therefore
\begin{equation*}
(b_{k+1}-1)\tau<(i_k-1)\tau+c,
\end{equation*}
which indicates
\begin{equation*}
|I_k|=b_{k+1}-i_k<\frac{c}{\tau}=\frac{1}{\tau}.
\end{equation*}

In other words, idle periods cannot be longer than $1/\tau$ since every $1/\tau$ slots there will be $\tau\cdot\frac{1}{\tau}=1$ packet due for transmission based on the debt based interpretation of the adaptive transmission policy.

The busy interval $B_k$ begins from slot $b_k$ (inclusion) and ends in slot $i_k-1$ (inclusion). According to Lemma \ref{lemma:structure1} it holds that $A_{b_k}<\tau b_k$ and
\begin{equation}\label{eqn:subtract}
A_{b_k}\geq A_{b_k-1}\geq\tau b_k-\tau.
\end{equation}
This can be interpreted as follows. In time slot $b_k$ the MS-AT policy has accumulated $A_{b_k}\in[\tau b_k-\tau,\tau b_k)$ successful transmissions, which is still short compared with the target $\tau b_k$.

By definition slot $i_k$ is the first slot after slot $b_k$ in which the MS-AT policy has no debt. Again according to Lemma \ref{lemma:structure1} it holds that $A_{i_k}\geq\tau i_k$ and
\begin{equation}\label{eqn:subtract2}
A_{i_k}\leq A_{i_k-1}+c<\tau (i_k-1)+c,
\end{equation}
\textit{i.e.}, the balance in slot $i_k$ is less than $c-\tau$.

Now consider the event $\{|B_k|=n\}$. It holds that
\begin{eqnarray}
\nonumber \{|B_k|=n\}&=&\{i_k=n+b_k\} \\
\label{eqn:set1}&\subset&\{A_{n+b_k}<\tau (n+b_k-1)+c\} \\
\label{eqn:set2}&\subset&\{A_{n+b_k}-A_{b_k}<\tau n+c\} \\
\label{eqn:set3}&=&\{\frac{A_{n+b_k}-A_{b_k}}{n}<\tau+\frac{c}{n}\},
\end{eqnarray}
where Eq. (\ref{eqn:set1}) and Eq. (\ref{eqn:set2}) make use of Eq. (\ref{eqn:subtract}) and Eq. (\ref{eqn:subtract2}), respectively.

By definition of busy interval the secondary user is always transmitting upon idle sensing results in the interval $[b_k,n-1+b_k]$. Therefore as $n$ tends to infinity the following convergence holds exponentially fast since the joint state which combines the channel state vector and the current channel index the secondary user is at evolves as a finite state recurrent Markov chain (see \cite{Chen&Tong:JSAC} Section IV.B),
\begin{eqnarray*}
\lim_{n\to\infty}\frac{1}{n}\mathbb{E}^{\tiny \mbox{MS}(\tau)}(A_{n+b_k}-A_{b_k})&=&\lim_{n\to\infty}\frac{1}{n}\mathbb{E}^{\tiny \mbox{MS}(\infty)}\sum_{t=1}^nR^{\tiny \mbox{MS}(\infty)}_t \\
>\Psi_X^{\tiny \mbox{MS}(\infty)}(\tilde{\theta})/\theta,
\end{eqnarray*}
where the middle term is the optimal throughput of the MS-AT policy with $\tau=\infty$, \textit{i.e.}, the secondary user always transmits upon idle sensing results.

Since $\Psi_X^{\tiny \mbox{MS}(\infty)}(\tilde{\theta})/\theta>\tau$, for $n$ large enough $\tau+c/n$ will lie to the left of $\Psi_X^{\tiny \mbox{MS}(\infty)}(\tilde{\theta})/\theta$, and thus to the left of $\lim_{n\to\infty}\frac{1}{n}\mathbb{E}^{\tiny \mbox{MS}(\infty)}\sum_{t=1}^nR^{\tiny \mbox{MS}(\infty)}_t$ with gap at least
$\lim_{n\to\infty}\frac{1}{n}\mathbb{E}^{\tiny \mbox{MS}(\infty)}\sum_{t=1}^nR^{\tiny \mbox{MS}(\infty)}_t-\Psi_X^{\tiny \mbox{MS}(\infty)}(\tilde{\theta})/\theta$.

Therefore the set in Eq. (\ref{eqn:set3}) corresponds to a deviation of at least $\lim_{n\to\infty}\frac{1}{n}\mathbb{E}^{\tiny \mbox{MS}(\infty)}\sum_{t=1}^nR^{\tiny \mbox{MS}(\infty)}_t-\Psi_X^{\tiny \mbox{MS}(\infty)}(\tilde{\theta})/\theta$ and has exponentially decaying probability in $n$, and so does the set $\{|B_k|=n\}$, which indicates that the lengths of the intervals $B_k$ have finite expectation.
\end{proof}

\begin{lemma}\label{lemma:finite}
For any $\tau<\Psi_X^{\tiny \mbox{MS}(\infty)}(\tilde{\theta})/\theta$, we have for the MS-AT policy with $\tau=\infty$ ($\mbox{MS}(\infty)$ policy)
\begin{equation}
\sup_{n}M^{\tiny \mbox{MS}(\infty)}_{R,n}(\theta)e^{-\theta\tau n}<\infty,
\end{equation}
where $M^{\tiny \mbox{MS}(\infty)}_{R,n}(\theta)=\mathbb{E}^{\tiny \mbox{MS}(\infty)}\exp(\theta\sum_{t=1}^nR^{\tiny \mbox{MS}(\infty)}_t)$.
\end{lemma}

\begin{proof}
Following Eq. (\ref{eqn:connection}) one can easily obtain that
\begin{equation}\label{eqn:limit}
\lim_{n\to\infty}\frac{1}{n}\log
M^{\tiny \mbox{MS}(\infty)}_{R,n}(\theta)=\Psi_X^{\tiny \mbox{MS}(\infty)}(\tilde{\theta}).
\end{equation}
Therefore for any $\tau<\Psi_X^{\tiny \mbox{MS}(\infty)}(\tilde{\theta})/\theta$, there
exists $\epsilon>0$ such that
$\Psi_X^{\tiny \mbox{MS}(\infty)}(\tilde{\theta})/\theta-\epsilon-\tau>0$. By the limit in
Eq. (\ref{eqn:limit}), we have that there exists $N$ such that for all
$n>N$
\begin{equation}
\frac{1}{n\theta}\log
M^{\tiny \mbox{MS}(\infty)}_{R,n}(\theta)>\Psi^{\tiny \mbox{MS}(\infty)}_X(\tilde{\theta})/\theta-\epsilon,
\end{equation}
which implies
\begin{equation}
M^{\tiny \mbox{MS}(\infty)}_{R,n}(\theta)<e^{n\theta(\Psi^{\tiny \mbox{MS}(\infty)}_X(\tilde{\theta})/\theta-\epsilon)},
\end{equation}
and further for all $n>N$
\begin{equation}
M^{\tiny \mbox{MS}(\infty)}_{R,n}(\theta)e^{-\theta\tau
n}<e^{n\theta(\Psi^{\tiny \mbox{MS}(\infty)}_X(\tilde{\theta})/\theta-\epsilon-\tau)}\leq
e^{N\theta(\Psi^{\tiny \mbox{MS}(\infty)}_X(\tilde{\theta})/\theta-\epsilon-\tau)}.
\end{equation}
Therefore we conclude
\begin{eqnarray*}
\sup_{n}M^{\tiny \mbox{MS}(\infty)}_{R,n}(\theta)e^{-\theta\tau n} &\leq&
\max\{e^{N\theta(\Psi^{\tiny \mbox{MS}(\infty)}_X(\tilde{\theta})/\theta-\epsilon-\tau)},
\\
&&\max_{1\leq i\leq N}\{M^{\tiny \mbox{MS}(\infty)}_{R,i}(\theta)e^{-\theta\tau i}\}\}<\infty,
\end{eqnarray*}
thus proving the lemma.
\end{proof}

\begin{lemma}
For any $\tau<\Psi_X^{\tiny \mbox{MS}(\infty)}(\tilde{\theta})/\theta$, the effective bandwidth of the MS-AT policy with parameter $\tau$ is
no less than $\tau$.
\end{lemma}

\begin{proof}
By the definition of effective bandwidth we have the effective
bandwidth of the secondary user is given by
\begin{equation}
a^{\tiny \mbox{MS}(\tau)}(\epsilon,b)=\lim_{n\to\infty}\frac{\log\mathbb{E}^{\tiny \mbox{MS}(\tau)}\exp(\theta\sum_{t=1}^nR^{\tiny \mbox{MS}(\tau)}_t)}{n\theta}.
\end{equation}
for buffer size $b\gg 1$ and $\theta=\log(\epsilon)/b$.

Define $T(n)$ to be the slot index that corresponds to the last slot before $n$ (inclusion) that is the start of an interval, \textit{i.e.},
\begin{equation}
T(n)=\max\{i\leq n:i=i_k\mbox{ or }i=b_k\mbox{ for some }k\}.
\end{equation}

Then the analysis deal with two possibilities, $T(n)=i_k$ for some $k$ and $T(n)=b_k$ for some $k$.

\begin{enumerate}
\item $T(n)=i_k$ for some $k$. Then we have
\begin{eqnarray*}
&&\mathbb{E}e^{\theta(\sum_{t=1}^nR_t-\tau n)} \\
&=& \mathbb{E}e^{\theta(A_{T(n)}-\tau T(n)+\sum_{k=T(n)}^{n}R_k-\tau(n-T(n)))} \\
&\leq& \mathbb{E}e^{\theta(\sum_{k=T(n)}^{n}R_k-\tau(n-T(n)))} \\
&\leq& \mathbb{E}e^{-\tau|I_k|\theta} \\
&\leq& \mathbb{E}e^{-\theta}
\end{eqnarray*}
where we used the inequalities $A_{i_k}\geq\tau i_k$ and $|I_k|<1/\tau$.
Therefore we have
\begin{eqnarray*}
\lim_{n\to\infty}\frac{1}{n}\log\mathbb{E}e^{\theta(\sum_{t=1}^nR_t-\tau
n)}\leq0,
\end{eqnarray*}
which leads to
\begin{eqnarray*}
\lim_{n\to\infty}\frac{1}{n\theta}\log\mathbb{E}e^{\theta\sum_{t=1}^nR_t}\geq\tau.
\end{eqnarray*}

\item $T(n)=b_k$ for some $k$. Then we have
\begin{eqnarray*}
&&\mathbb{E}e^{\theta(\sum_{t=1}^nR_t-\tau n)} \\
&=& \mathbb{E}e^{\theta(A_{T(n)}-\tau T(n)+\sum_{k=T(n)}^{n}R_k-\tau(n-T(n)))} \\
&\leq& \mathbb{E}e^{\theta(-\tau+\sum_{k=T(n)}^{n}R_k-\tau(n-T(n)))} \\
&\leq& e^{-\theta\tau}\sup_k M_{R,k}(\theta)e^{-\theta\tau k}
\end{eqnarray*}
where we used the inequality $A_{b_k}-\tau
b_k\geq-\tau$ and Lemma \ref{lemma:finite}.

Therefore we have
\begin{eqnarray*}
\lim_{n\to\infty}\frac{1}{n}\log\mathbb{E}e^{\theta(\sum_{t=1}^nR_t-\tau
n)}\leq0,
\end{eqnarray*}
which also leads to
\begin{eqnarray*}
\lim_{n\to\infty}\frac{1}{n\theta}\log\mathbb{E}e^{\theta\sum_{t=1}^nR_t}\geq\tau.
\end{eqnarray*}
\end{enumerate}

Therefore we have proved that the effective bandwidth of the MS-AT
policy is at least as large as the parameter $\tau$.
\end{proof}

\begin{lemma}
The throughput of the MS-AT policy with parameter $\tau$ is at most $\tau$.
\end{lemma}

\begin{proof}
The throughput of the secondary user is given by $\lim_{n\to\infty}\frac{1}{n}\mathbb{E}^{\tiny \mbox{MS}(\tau)}\sum_{t=1}^nR^{\tiny \mbox{MS}(\tau)}_t$.
Again split the analysis into two possibilities, $T(n)=i_k$ for some $k$ and $T(n)=b_k$ for some $k$.

\begin{enumerate}
\item $T(n)=i_k$ for some $k$. Then we have
\begin{eqnarray}
\nonumber&&\mathbb{E}(\sum_{t=1}^nR_t-\tau n) \\
\nonumber&=& \mathbb{E}(A_{T(n)}-\tau T(n)+\sum_{k=T(n)}^{n}R_k-\tau(n-T(n))) \\
\label{eqn:notx}&\leq& \mathbb{E}(c-\tau),
\end{eqnarray}
where Eq. (\ref{eqn:notx}) is due to the fact that from slot $T(n)$ to slot $n$ the secondary user stays silent. Thus it holds that
\begin{equation*}
\lim_{n\to\infty}\frac{1}{n}\mathbb{E}(\sum_{t=1}^nR_t-\tau n)\leq0,
\end{equation*}
which leads to
\begin{equation*}
\lim_{n\to\infty}\frac{1}{n}\mathbb{E}\sum_{t=1}^nR_t\leq\tau.
\end{equation*}

\item $T(n)=b_k$ for some $k$. Then we have
\begin{eqnarray*}
&&\mathbb{E}(\sum_{t=1}^nR_t-\tau n) \\
&=& \mathbb{E}(A_{T(n)}-\tau T(n)+\sum_{k=T(n)}^{n}R_k-\tau(n-T(n))) \\
&\leq& \mathbb{E}\sum_{k=T(n)}^{n}c \\
&\leq& c\mathbb{E}|B_k|.
\end{eqnarray*}

Due to the fact that the expected length of a busy period is finite, we have in parallel with the previous case
\begin{equation*}
\lim_{n\to\infty}\frac{1}{n}\mathbb{E}(\sum_{t=1}^nR_t-\tau n)\leq0,
\end{equation*}
which leads to
\begin{equation*}
\lim_{n\to\infty}\frac{1}{n\theta}\mathbb{E}\sum_{t=1}^nR_t\leq\tau,
\end{equation*}
\end{enumerate}

Therefore we have proved that the throughput of the MS-AT policy is
at most as large as the parameter $\tau$.
\end{proof}

\begin{corollary}\label{col:1}
For any $\tau<\Psi_X^{\tiny \mbox{MS}(\infty)}(\tilde{\theta})/\theta$, the effective bandwidth and throughput of MS-AT with parameter $\tau$ is both $\tau$.
\end{corollary}

\begin{proof}
Application of Jensen's inequality implies that for the same opportunistic access policy $\pi$ the effective
bandwidth is no greater than the throughput
$\lim_{n\to\infty}\frac{1}{n}\mathbb{E}_\pi\sum_{t=1}^nR^\pi_t$. However,
we have already established that the throughput of MS-AT is no greater than
the parameter $\tau$ and the effective bandwidth is no less than the
parameter $\tau$. Therefore combining the two facts, we conclude that the throughput and the effective bandwidth are both equal to
the parameter $\tau$.
\end{proof}

\begin{lemma}
For any collision parameter $\gamma$, the following choice of parameter
$\tau$
\begin{equation}\label{eqn:throughputgiven}
\tau=\frac{N\gamma\exp(-\lambda T)(1-v(0)\exp(-\lambda
T))}{1-\exp(-\lambda
T)}
\end{equation}
is feasible.
\end{lemma}

\begin{proof}
The MS-AT policy only transmits upon idle sensing results. Therefore for any horizon $n$ the ratio between the expected number of total successful transmissions and the expected number of total collisions is $\frac{1-\exp(-\lambda
T)}{\exp(-\lambda T)}$. With MS-AT policy the total number of successful transmissions the secondary user achieves during the first $t$ slots satisfies
\begin{equation}\label{eqn:lessthan}
A_{t+1}<\tau t+c.
\end{equation}
Taking the expectation of Eq. (\ref{eqn:lessthan}) implies that the expected number of total successful transmissions is upper bounded by $\tau t+c$ and thus the expected number of total collisions is bounded by $\frac{(\tau t+c)(1-\exp(-\lambda T))}{\exp(-\lambda
T)}$. Then we divide this upper bound expression by $t$, take the limit $t\to\infty$ and take into account the scaling coefficient $1-v(0)\exp(-\lambda
T)$ (see Section \ref{sec:collision}). Now we obtain that the sum of the collision $C_{{\tiny \mbox{MS}(\tau)},i}\leq N\gamma$.

Note that the channel sensing policy and the transmission policy is symmetric with respect to the $N$ channels. Therefore the lemma is proved.
\end{proof}

Then we are ready to prove Theorem \ref{thm:tightthm}.

\begin{proof}
Note that the parameter $\tau(\gamma)$ given by Eq. (\ref{eqn:throughputgiven}) is no less
than the best throughput for any opportunistic access policy under collision limit $\gamma$ (see \cite{Chen&Tong:JSAC}).

We also know from Corollary \ref{col:1} that
for any $\tau<\Psi_X^{\tiny \mbox{MS}(\infty)}(\tilde{\theta})/\theta$ the effective bandwidth and throughput of MS-AT with parameter $\tau$ is both $\tau$.
Therefore for any $\tau<\Psi_X^{\tiny \mbox{MS}(\infty)}(\tilde{\theta})/\theta$ the effective bandwidth of MS-AT with parameter $\tau$ matches with the optimal throughput upper bound, thus optimal with respect to effective bandwidth.

Combining this fact with Lemma \ref{lemma:mono} and the fact that MS-AT$(\infty)$ has optimal effective bandwidth $\Psi_X^{\tiny \mbox{MS}(\infty)}(\tilde{\theta})/\theta$ for loose collision parameters, we conclude that for any
collision parameter $\gamma$ the parameter $\tau$ given by
\begin{equation*}
\tau=\frac{N\gamma\exp(-\lambda T)(1-v(0)\exp(-\lambda
T))}{1-\exp(-\lambda
T)}
\end{equation*}
achieves the optimal effective bandwidth
\begin{equation*}
\mbox{\textmd{EB}}^\ast=\min\{\frac{N\gamma\exp(-\lambda T)(1-v(0)\exp(-\lambda
T))}{1-\exp(-\lambda
T)},\frac{\Psi_X^{\tiny \mbox{MS}(\infty)}(\tilde{\theta})}{\theta}\}.
\end{equation*}

The throughput optimality of the MS-AT policy can be proved similarly.
\end{proof}

\section{Simulation Results}\label{sec:simu}

We plot the throughput and effective bandwidth versus the collision parameter
$\gamma$ for MS-AT. In comparison, we also plot the throughput and
effective bandwidth versus the collision parameter
$\gamma$ for MS-MT, the policy with myopic sensing but memoryless transmission.
In the simulation the number of primary channel is set to be $N=2$.
The slot length in the simulation is taken to be
$T=0.25\textmd{ms}$. The channel parameters we use are $\mu=1/2\textmd{ms}^{-1}$ and
$\lambda=1/3\textmd{ms}^{-1}$. The effective bandwidth parameter $\theta=\log(\epsilon)/b$ is taken to be -0.08.

Fig. \ref{fig:Two_channels} depicts the throughput and effective bandwidth versus the collision parameter $\gamma$ for
the MS-AT and MS-MT policies. The plot shows that the effective bandwidth and throughput
matches up to a certain point for MS-AT policy, and then levels up at optimal effective bandwidth for loose collision parameter, validating the optimality result for MS-AT policy.
On the other hand, the throughput and effective bandwidth for
the MS-MT policy deviates in the entire $\gamma$-axis, showing that the MS-MT policy is indeed suboptimal when the collision constraint is tight. On the other hand, when the collision constraint is loose, the two policies both give the optimal effective bandwidth. The optimal effective bandwidth deviates from the throughput in the loose collision regime, as predicted in the presentation.
Also note that in Fig. \ref{fig:Two_channels} the optimal throughput as well as the effective bandwidth is first
linearly increasing in $\gamma$ and then levels up.

\begin{figure}[t]
\centering \subfigure[MS-AT] {
 \label{fig:Two_channels:a}
 \includegraphics[width=1.5in]{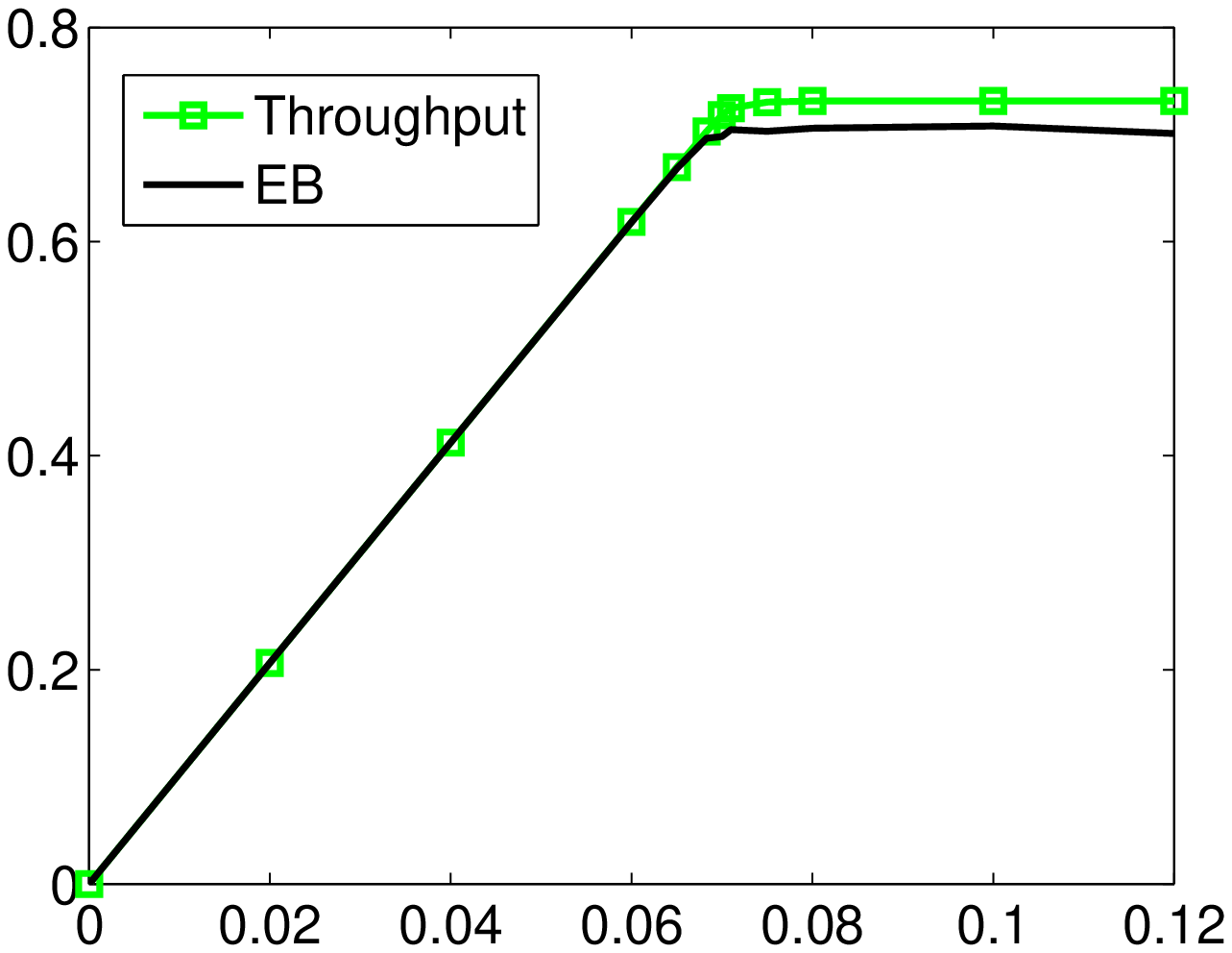}
} \subfigure[MS-MT] {
 \label{fig:Two_channels:b}
 \includegraphics[width=1.5in]{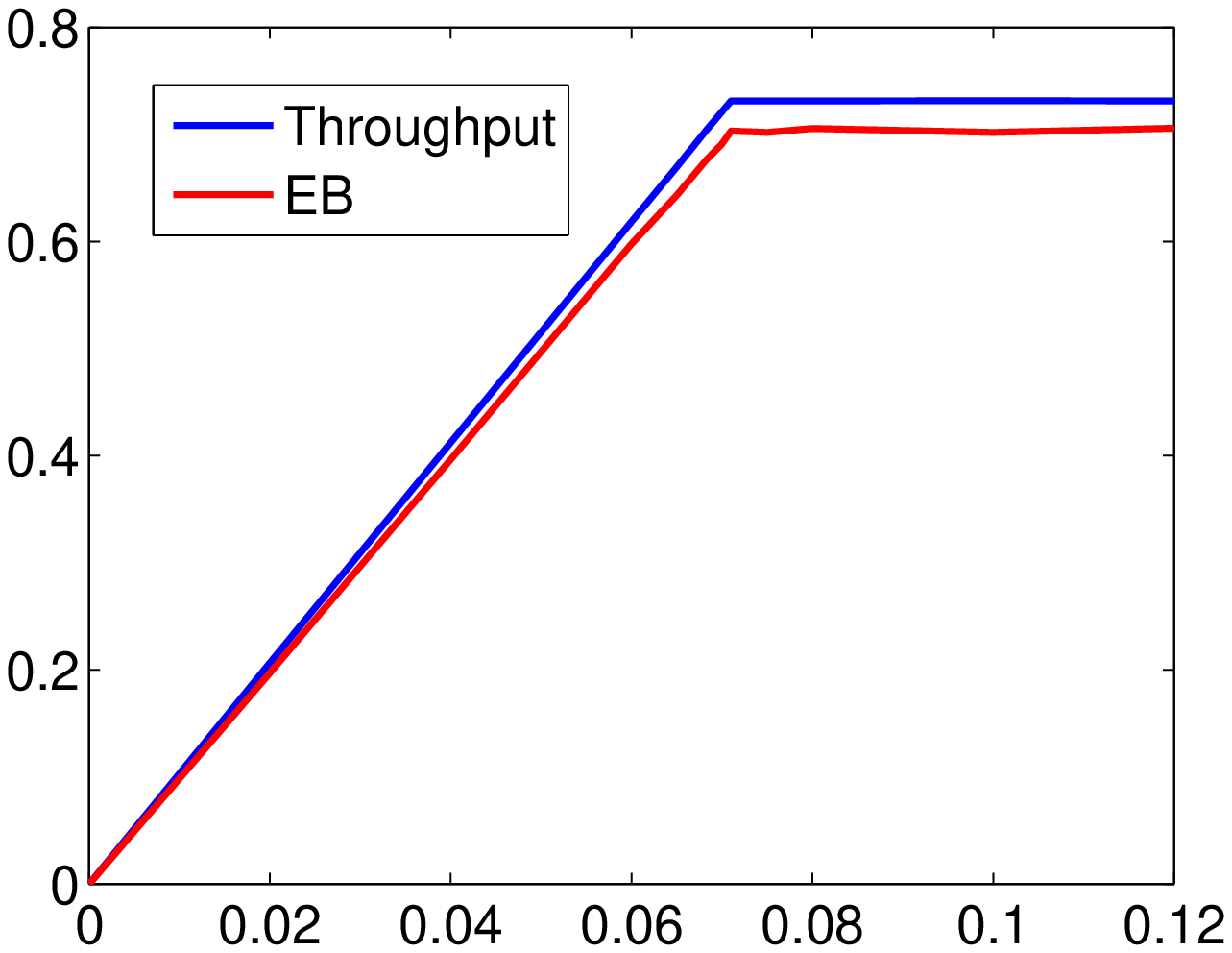}
}\caption{Throughput and Effective bandwidth of MS-AT and MS-MT.}
\label{fig:Two_channels}
\end{figure}

The empirical distribution of $A_t$ is plotted in Fig. \ref{fig:Two_channel}. It can be observed that under MS-AT policy, the distribution of $A_t$ is more condensed near the target $\tau t$, whereas under MS-MT policy, the distribution of $A_t$ is more far apart from $\tau t$.

\begin{figure}[t]
\centering \subfigure[MS-AT] {
 \label{fig:Two_channel:a}
 \includegraphics[width=1.5in]{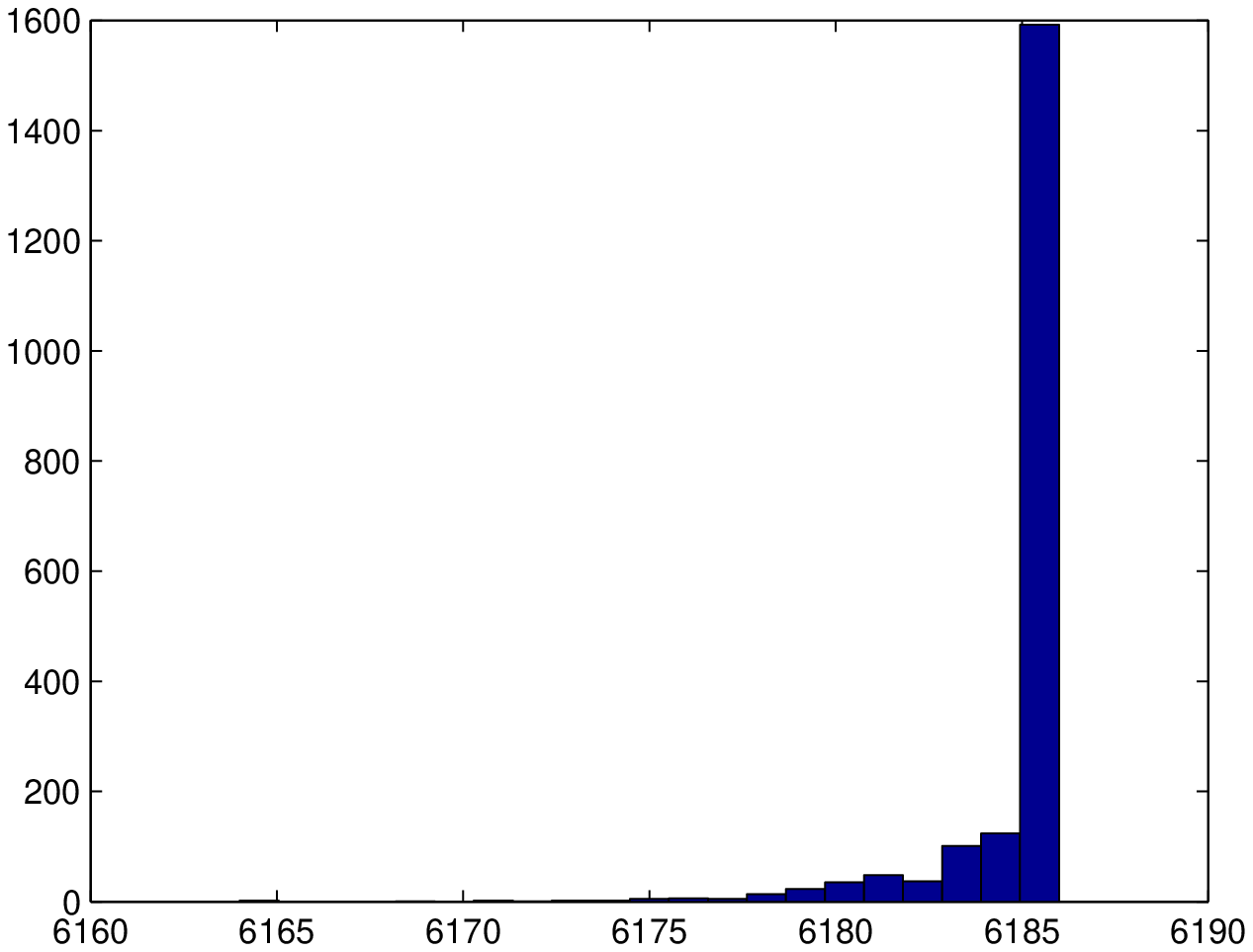}
} \subfigure[MS-MT] {
 \label{fig:Two_channel:b}
 \includegraphics[width=1.5in]{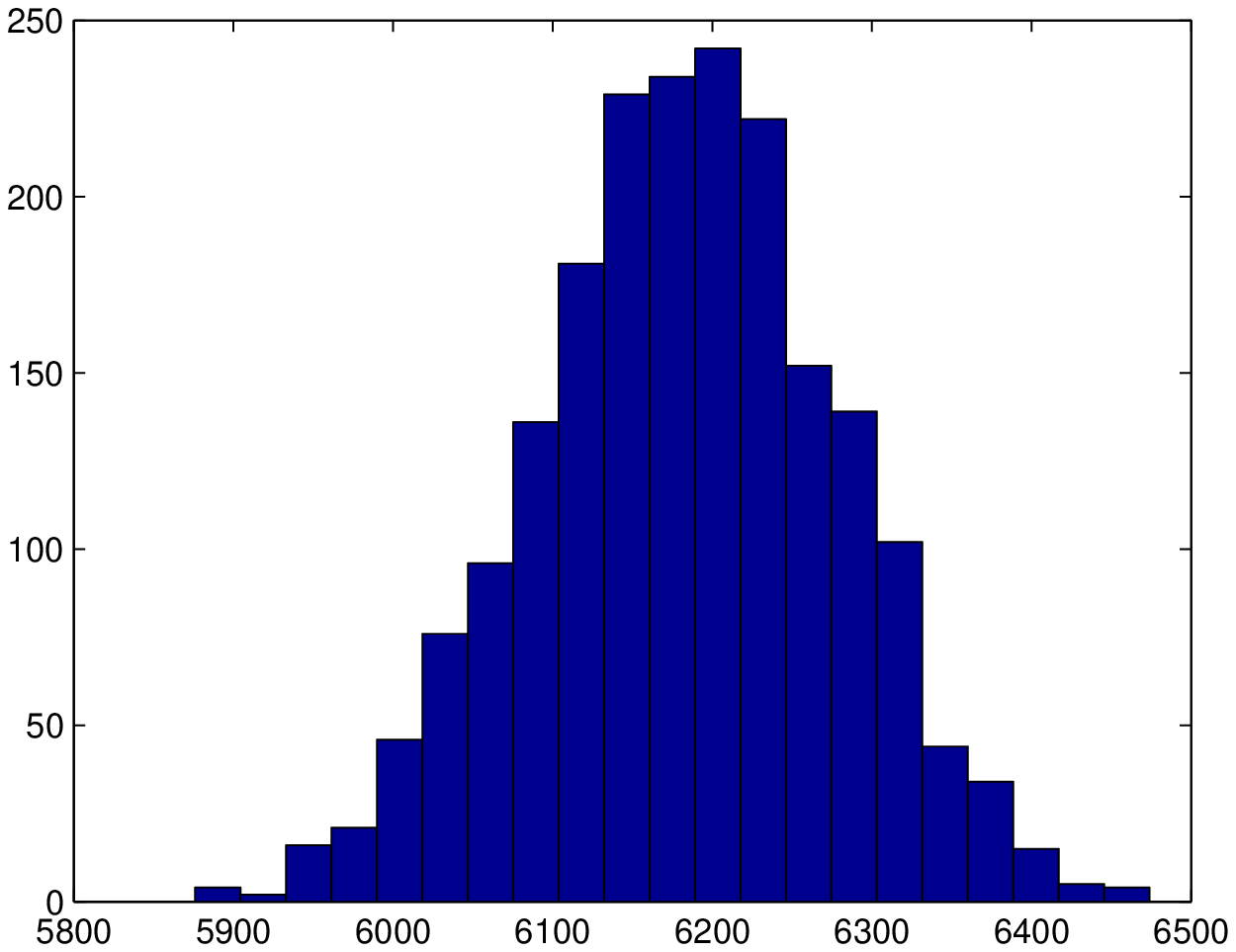}
}\caption{Empirical distribution of MS-AT and MS-MT.}
\label{fig:Two_channel}
\end{figure}

\section{Conclusions}

Multichannel opportunistic access of homogeneous continuous time Markov
channels is considered. The objective is to design an opportunistic access policy
that maximizes the effective bandwidth under the interference
constraints. This paper shows that a myopic sensing plus adaptive
transmission policy is optimal with respect to both effective bandwidth and throughput under all levels of collision constraint. The
optimality result may find applications in cognitive radio networks
for spectrum overlay with QoS requirement. Although we focus on single user scenario in this paper, the technique can be extended to multiuser scenario and similar results can be established for the multiuser effective bandwidth region.

\bibliographystyle{ieeetr}
{\small
\bibliography{INFOCOM12}
}
\end{document}